\documentclass{amsart}
\usepackage{amsmath, amssymb, amsthm, amscd, amsfonts, eucal, hyperref}
\usepackage[all]{xy}

\newtheorem{theorem}{Theorem}[section]
\newtheorem{proposition}[theorem]{Proposition}
\newtheorem{lemma}[theorem]{Lemma}
\newtheorem {corollary}[theorem]{Corollary}
\theoremstyle {definition}
\newtheorem {definition}[theorem]{Definition}
\newtheorem {example}[theorem]{Example}
\theoremstyle {remark}
\newtheorem{remark}[theorem]{Remark}

\def\grad{\operatorname{grad}}

\def\supp{\mathrm{supp}}

\dedicatory{Dedicated to Professor Ha Huy Vui on the occasion of his sixtieth anniversary.}
\begin{document}

\title{On the topology of rational functions in two complex variables}
\author{Nguyen Tat Thang}

\address{Nguyen Tat Thang,
Institute of Mathematics, 18 Hoang Quoc Viet road, 10307 Hanoi, Vietnam.}
\email{ntthang@math.ac.vn}

\begin{abstract}
 We give some characterizations for the critical values at infinity of a rational function in two complex variables in terms of the Euler characteristic, the Malgrange condition and the M-tameness.
\end{abstract}
\maketitle


\section{Introduction}
Let $F$ be a rational function in $n$ complex variables. It is well-known that $F$ is a locally trivial fibration outside some finite subsets of $\mathbb C$ (see [T1]). The smallest such subset is called the bifurcation value set of $F$ and is denoted by $B(F)$. A natural question is how to compute this set $B(F)$.

We recall the definition of the so-called critical values at infinity (or atypical values) of a rational function.
\begin{definition}
A value $t_0\in\mathbb C$ is called a {\it regular value at infinity} of $F$ if there is a positive real number $\delta>0$ and a compact subset $K\subset \mathbb C^n$ such that the restriction
$$F: F^{-1}(D_\delta(t_0))\setminus K \rightarrow D_\delta(t_0),$$
is a $C^\infty$-trivial fibration, where $D_\delta(t_0):=\{t\in \mathbb C: |t-t_0|<\delta\}$. 

If $t_0\in \mathbb C$ is not a regular value at infinity, we call it a {\it critical value at infinity} (or {\it atypical value}) of the rational function $F$. Denote the set of those critical values at infinity of $F$ by $B_\infty(F)$.
\end{definition}

Obviously $B(F)$ contains the set $K_0(F)$ of the critical values and the set $B_\infty(F)$, $B(F)\supseteq B_\infty(F)\cup K_0(F)$.

The aim of this article is to study the sets $B(F)$ and $B_\infty(F)$ of a rational function $F$ in two complex variables, that is, the case $n=2$.

Let $f, g\in \mathbb C[x, y]$ be two non-zero polynomials without common factors and set $F=f/g$. If $\deg(g)=0$, or equivalently $F$ is a polynomial function, then one can prove that $B(F)=B_\infty(F)\cup K_0(F)$. Our first result is the following theorem.

\begin{theorem}\label{thm12} Assume that $\deg(f)>\deg(g)$. We have
$$B(F)=B_\infty(F)\cup K_0(F)\cup K_1(F).$$
\end{theorem}
Here $K_1(F)$ is a subset of $\mathbb C$ which is defined in Section 2 and by using the Milnor number at in-determinant point of $F$.

Our second result is about the set of critical values at infinity $B_\infty(F)$. In the case of polynomial functions, there has been several interesting characterizations of this set. One of those via Euler characteristic is due to H\`a Huy Vui - L\^e D\~ung Tr\'ang [HL] and Suzuki [S1].
\begin{theorem}{\rm (\cite{HL}, \cite{S})}\label{thm13}
Let $F$ be a polynomial function in two complex variable and $t_0\in \mathbb C$ be a regular value of $F$. Then $t_0\in  B_\infty(F)$ if and only if the Euler characteristic of the fiber $F^{-1}(t)$ is not a constant in every neighborhood of $t_0$.
\end{theorem}

Another characterization of the set $B_\infty(F)$ is by the Fedoryuk condition, Malgrange condition and the M-tameness. Recall that for a rational function $F$, we denote by $\widetilde K_\infty(F)$ the set of $t\in \mathbb C$ such that there exists a sequence $\{x_k\}_k\subset \mathbb C^n$, $x_k\to \infty$, such that $F(x_k)\to t$ and $||\mathrm{grad} F(x_k)||\to 0$. We say that $F$ satisfies {\it Fedoryuk condition} at a value $t\in \mathbb C$ if $t\not\in\widetilde K_\infty(F)$. If, in addition, we require that $||x_k||.||\mathrm{grad} F(x_k)||\to 0$, then we get a subset $K_\infty(F)\subseteq \widetilde K_\infty(F)$. We say that $F$ satisfies {\it Malgrange condition} at a value $t\in \mathbb C$ if $t\not\in K_\infty(F)$.

Let $M_\infty(F)$ denote the set of values $t\in\mathbb C$ such that there are sequences $\{\lambda_k\}_k\subset \mathbb C$ and $\{x_k\}_k\subset \mathbb C^n$ with $x_k\to \infty$, $F(x_k)\to t$ and $\mathrm{grad} F(x_k)=\lambda_k x_k$ for all $k=0, 1, 2, \ldots$. We say that the rational function $F$ is M-tame at a value $t\in \mathbb C$ if $t\not\in M_\infty(F)$. The set $M_\infty(F)$ of a rational function has been studied in a recent paper of Arnaud Bodin and Anne Pichon where they prove that a non-zero value $t_0\in \mathbb C$ is in $M_\infty(F)$ if and only if outside a large compact  set of $\mathbb C^2$, the topological type of the curve $F^{-1}(t)$ is a constant for all $t$ near $t_0$. 

For the case of polynomial functions in two variables, we have the following characterizations of the set $B_\infty(F)$ due to H\`a Huy Vui and Ishikawa.
\begin{theorem}{\rm (\cite{H}, \cite{I})}\label{thm14}
Let $F:\mathbb C^2\rightarrow \mathbb C$ be a polynomial function and $t\in \mathbb C$. The followings are equivalent:

(i) $t\in B_\infty(F)$;

(ii) $t\in \widetilde K_\infty(F)$;

(iii) $t\in K_\infty(F)$;

(iv) $t\in M_\infty(F)$.
\end{theorem}

Our second result in this article is a generalization of Theorems \ref{thm13} and \ref{thm14} to the case of rational functions. We will show that under some wild assumptions, the critical values at infinity of rational functions in two variables can be determined in terms of the Euler characteristic, the Malgrange condition and the M-tameness (Theorems \ref{thrm2.1}, \ref{thrm2.2} and \ref{main2}). Moreover, we give examples showing that the Fedoryuk condition can not characterize the critical values at infinity of those functions.

The article consists of three sections. Theorem \ref{thm12} is proved in Section 2. Section 3 is devoted to a generalization of Theorems \ref{thm13} and \ref{thm14} for rational functions in two complex variables as mentioned above. The main results of this section are Theorems \ref{thrm2.1}, \ref{thrm2.2} and \ref{main2}. 


\section{The bifurcation set}
In this section we give some descriptions for the set of bifurcation values of a rational function in two complex variables.

Let $F= f/g: \mathbb{C}^2\setminus\{g= 0\}\to \mathbb{C}$ be a rational function, where $f, g\in \mathbb{C}[x, y]$ have no common factor. Let
$$A(F):=\{(x, y)\in \Bbb{C}^2 : f(x, y)=g(x, y)=0\}.$$
For each $t\in \Bbb{C}$ set $$d_t:= \deg (f- tg),$$
$$ V_t:= \{(x, y)\in \mathbb{C}^2: f(x, y)-tg(x, y) =0\},$$
$$G(x, y, z, t):= z^{d_t}f(x/z, y/z) - tz^{d_t}g(x/z, y/z)$$
 and $$\overline{V_t}:= \{[x : y : z]\in \mathbb{C}P^2 : G(x, y, z, t)=0\}.$$ 
Let $V^t_{\infty}= \overline{V_t}\cap H_{\infty}$ be set of points at infinity of $\overline{V_t}$.
\begin{remark}{\rm
(i) $A(F)$ contains finitely many points.

(ii) For all $t\in \Bbb{C}$ we have $A(F)\subset \{f-tg=0\}.$ Moreover, if $t_0$ is a regular value of $F$ and $t$ is near $t_0$ enough then every point $p\in A(F)$ is either a regular point or an isolated singular point of the curve $V_t$.
}
\end{remark}

\begin{definition}{\rm
We denote by $K_1(F)$ the set of $t_0\in \mathbb{C}\setminus K_0(F)$ such that there exist $p\in A(F)$ and $\mu_p(f-t_0g)\neq \mu_p(f-tg)$ for all $t\neq t_0$ near $t_0$ enough, where $\mu_p(f-tg)$ is the Milnor number of $f-tg$ at $p$.
}
\end{definition}

\begin{remark}{\rm For each curve $V\subset \mathbb{C}^2$ we denote by $\textrm{Sing}V$ the set of singular points of $V$. Then $t_0\notin K_1(F)$ if $\textrm{Sing}\{f-t_0g=0\} \cap A(F) = \emptyset$.
}\end{remark}

\begin{lemma}\label{thrmbifur-huuty1}
Let $F:= \frac{f}{g}\colon \Bbb{C}^2\setminus\{g= 0\}\to \Bbb{C}$ be a rational function, where $f, g\in \Bbb{C}[x, y]$ have no common factor. Then
$$B(F)\subset K_0(F)\cup B_{\infty}(F)\cup K_1(F).$$
\end{lemma}

To prove the Lemma, we need the following results.

\begin{theorem}{\rm (\cite{Timourian1977, LR})}\label{thrmLe}
Let $g_s\colon (\Bbb{C}^n, 0)\to (\Bbb{C}, 0), n\neq 3, s\in \mathbb{R}^m$ be a differential family of holomorphic germs such that for all $s$ the origin $0\in \Bbb{C}^n$ is an isolated singular point of $g_s$. Assume that $\mu_0(g_s)=\mu_0(g_0)$ for all $s$ near $0\in \mathbb{R}^m$ enough. Then, there exist a neighborhood $D$ of  $0\in \Bbb{R}^m$, a small ball $B$ centered at the origin and a continuous family of homeomorphisms
$$\Phi_s\colon g_s^{-1}(0)\cap B\to g_0^{-1}(0)\cap B, s\in D$$ 
such that $\Phi_s(0)=0$.
\end{theorem}

Let $h : X \to Y$ be a continuous map. A {\it homotopy} of $h$ is a continuous map $H : X\times [0; 1] \to Y$ such that $H(x; 0) = h(x)$ for all $x\in X$.

\begin{definition}{\rm
The continuous map $\pi: E \to B$ is called a {\it fibration}, or equivalently, {\it has homotopy lifting property}, if for all polytopes $X$ and for any continuous map $h : X \to E$, every homotopy $\Phi $ of $\pi \circ  h$ can be lifted to a homotopy of $h$, i.e. there exists a homotopy $H$ of $h$ such that the diagram
$$ \xymatrix{
   &E\ar[d]^{\pi}\\
   X\times [0; 1]\ar[ur]^{H}\ar[r]_{\Phi} &B}$$
 commutes.
}
\end{definition}

\begin{definition}{\rm (\cite{M})
Let $X, Y$ be topology spaces. Two homotopies
$$H, H^{'}: X\times [0, 1] \to Y$$
are called to {\it have the same germ} if they coincide in a neighborhood of $X \times \{0\}$.
}
\end{definition}

\begin{definition}{\rm (\cite{M})
The continuous map $\pi: E \to B$ is called a {\it homotopic submersion}, or equivalently say that it has the
{\it germ-of-homotopy lifting property}, if for every polytope $X$ and every continuous map $h : X \to E$ every germ-of-homotopy of $\pi \circ h$ lifts to a germ-of-homotopy for $h$.
}
\end{definition}

\begin{definition}{\rm (\cite{M})
The continuous map $\pi: E \to B$ is called a {\it local homotopic submersion} if for every $x\in E$ there is a neighborhood $U(x)\subset E$ such that the restriction $\pi_{|U(x)}$ is a homotopic submersion from $U(x)$ onto $\pi(U(x))$.
}
\end{definition}

\begin{lemma}\label{lm1.1}{\rm (\cite{M}, Lemma 6)} Let $\pi: E \to B$ be an open  continuous map. Assume that $\pi $ is a local homotopic submersion. Then $\pi$ is a homotopic submersion.
\end{lemma}

It deduces from Lemma \ref{lm1.1} that
\begin{lemma}\label{lm1.2}
Let $f:V_1 \to V_2$ be a differential map. Assume that $f$ is a submersion. Then $f$ is a homotopic submersion.
\end{lemma}

\begin{lemma}\label{lm1.3}{\rm (\cite{M})}  In the following commutative diagram of continuous maps,

\begin{displaymath}
\xymatrix{
E \ar[dr]^{\pi}  \ar[r]^{h}  &  E^{'}  \ar[d]^{\pi^{'}}  \\
           &  B }
\end{displaymath}
assume that $\pi $ and $\pi^{'} $ are surjective homotopic submersions. If $\pi^{'}$ is a fibration and  for every $b\in B$, the restriction $h_b:=h_{|\pi^{-1}(b)}: \pi^{-1}(b)\to  \pi^{'-1}(b)$ is a weak homotopy equivalence, then $\pi$ is a fibration.
\end{lemma}

\begin{lemma}{\rm (\cite{M}, Corollary 32)}\label{lm1.4}
Let $\pi: E \to B$ be a differential map such that $\dim_{\Bbb{R}} E= \dim_{\Bbb{R}} B + 2$. Assume that the followings are satisfied:
\begin{enumerate}
\item[(i)] $\pi$  is a surjective;
\item[(ii)] $\pi$ is a fibration;
\item[(iii)]  $\pi $ is a submersion.
\end{enumerate}
Then $\pi$ is a locally $C^{\infty}-$trivial fibration.
\end{lemma}

\begin{proof}[Proof of Lemma \ref{thrmbifur-huuty1}] Let $t_0\notin K_0(F)\cup B_{\infty}(F)\cup K_1(F)$. Without loss of generality, we may assume that $t_0=0.$

Since $0\notin K_1(F)$ there is a neighbourhood $D_1$ of $0$ such that
$$\mu_p(f-tg)=\mu_p(f), \forall t\in D_1, p\in A(F).$$
Hence, according to Theorem \ref{thrmLe}, there exists a small neighbourhood $D_2$ of $0$ such that  for every $p\in A(F)$, there is a ball $B(p)$ centered at $p$ and there is a continuous family of homeomorphisms
$$\Phi_p(s)\colon \{f-sg=0\}\cap B(p)\to \{f=0\}\cap B(p), s\in D_2$$
such that $\Phi_p(s)(p)=p.$

The family $\Phi_p(s), s\in D_2$, generates a homeomorphism as follows
\begin{align*}
\Phi_p\colon F^{-1}(D_2)\cap B(p)&\to (F^{-1}(0)\cap B(p)) \times D_2\\
\Phi_p(x) &= (\Phi_p (F(x)), F(x)).
\end{align*}

Thus, we have the following commutative diagram of continuous maps
\begin{displaymath}
\xymatrix{
F^{-1}(D_2)\cap B(p) \ar[dr]^{F}  \ar[r]^{\Phi_p}  &  (F^{-1}(0)\cap B(p)) \times D_2.  \ar[d]^{pr_2}  \\
           &  D_2 }
\end{displaymath}

Since $0\notin K_0(F)$, we can choose $D_2$ small enough such that the restriction
$$F_{|F^{-1}(D_2)\cap B(p)}: F^{-1}(D_2)\cap B(p)\to D_2$$
is a submersion, hence, is a homotopic submersion (according to Lemma \ref{lm1.2}). Moreover, since $\Phi_p(s)$ is a homeomorphism, the restriction
$$(\Phi_p)_{|F^{-1}(s)}: F^{-1}(s)\to pr_2^{-1}(s)= (F^{-1}(0)\cap B(p)) \times \{s\}$$ 
is also a homeomorphism. Therefore, it is a weak homotopy equivalence. Thus, by Lemma \ref{lm1.3}, the map
$$F_{|F^{-1}(D_2)\cap B(p)}: F^{-1}(D_2)\cap B(p)\to D_2$$
is a fibration. It is easy to check that fibers of the map are two dimensional. It follows from Lemma \ref{lm1.4} that it is a $C^{\infty}-$trivial fibration. So there is a differmorphism
$$\Psi^1_p: F^{-1}(D_2)\cap B(p)\to (F^{-1}(0)\cap B(p)) \times D_2$$
 such that the following diagram
\begin{displaymath}
\xymatrix{
F^{-1}(D_2)\cap B(p) \ar[dr]^{F}  \ar[r]^{\Psi^1_p}  &  (F^{-1}(0)\cap B(p)) \times D_2  \ar[d]^{pr_2}\\
           &  D_2 }
\end{displaymath}
commutes.

On the other hand, since $0\notin B_{\infty}(F)$, there exist a neighbourhood $D_3$ of $0$, a compact set $B\subset \Bbb{C}^2$ and a homeomorphism
$$\Psi_2: F^{-1}(D_3)\setminus B\to (F^{-1}(0)\setminus B) \times D_3$$
such that the diagram 
\begin{displaymath}
\xymatrix{
F^{-1}(D_3)\setminus B \ar[dr]^{F}  \ar[r]^{\Psi^2}  &   (F^{-1}(0)\setminus B) \times D_3 \ar[d]^{pr_2}  \\
           &  D_2 }
\end{displaymath}
commutes.

Without loss of generality, we may assume that $D_2 = D_3= D$. Since $0$ is a regular value of $F$, we can choose $D$ small enough such that every $t\in D$ is regular value of $F$.

Now,  we construct a convenient vector field $v(x)$ on $F^{-1}(D)$ trivializing the restriction $F_{|F^{-1}(D)}$. Let $x^{\alpha}\in F^{-1}(D)$ arbitrary, we consider the following cases.

a) Case 1: $x^{\alpha}\in B(p)\cap F^{-1}(D), p\in A(F)$. Let $U^{\alpha}:= B(p)\cap F^{-1}(D)$ and
 $v^{\alpha}(x):= \frac{\partial (\Psi^1_p)^{-1}}{\partial s}(\Psi^1_p(x)),$
where $s$ is coordinate on $D$.

It is easy to verify that 
$$<v^{\alpha}(x), \grad F(x)> = 1, x\in U^{\alpha}.$$

b) Case 2: $x^{\alpha}\in F^{-1}(D)\setminus B$. Let $U^{\alpha}=F^{-1}(D)\setminus B$ and $v^{\alpha}(x)= \frac{\partial (\Psi^2)^{-1}}{\partial s}(\Psi^2(x)), x\in U^{\alpha}.$
Similarly, we have
$$<v^{\alpha}(x), \grad F(x)> = 1, x\in U^{\alpha}.$$

c) Case 3: $x^{\alpha}\in  (F^{-1}(D)\cap \textrm{int} B)\setminus (\cup_{p\in A(F)} \overline{B(p)})$. Let $U^{\alpha}:= (F^{-1}(D)\cap \textrm{int} B)\setminus (\cup_{p\in A(F)} \overline{B(p)}).$
Since $t\in D$ is regular value of $F$ then $\grad F(x)\neq 0$ for all $x\in F^{-1}(D).$ Let
$$v^{\alpha}(x) = \frac{\grad F(x)}{\|\grad F(x)\|}, x\in U^{\alpha}.$$
 
Then
$$<v^{\alpha}(x), \grad F(x)> = 1, x\in U^{\alpha}.$$

d) Case 4: $x^{\alpha}\in \partial B\cup (\cup_{p\in A(F)}\partial B(p)\cap F^{-1}(D))$. Since $t\in D$ is a regular value of $F$, then $x^{\alpha}$ is a regular point of $F$. Hence there is a small neighbourhood $U^{\alpha}$ of $x^{\alpha}$ such that $\grad F(x)\neq 0$ for all $x\in U^{\alpha}.$ Similarly, let $$v^{\alpha}(x) = \frac{\grad F(x)}{\|\grad F(x)\|}, x\in U^{\alpha}.$$

Let $\lambda^{\alpha}$ be a smooth unit partition on $F^{-1}(D)$ such that $\supp \lambda^{\alpha}\subset U^{\alpha}$ for all $\alpha$. The vector field $v(x)$ on $F^{-1}(D)$ is defined by $v(x):= \sum \lambda^{\alpha}(x) v^{\alpha}(x).$

It is clear that $v(x)$ is smooth and satisfies
$$<v(x), \grad F(x)> = 1, x\in F^{-1}(D).$$

By integrating the vector field $v(x)$, we get the diffeomorphism trivializing the map
$$F_{|F^{-1}(D)}: F^{-1}(D)\to D.$$
Thus $0\notin B(F).$ 
\end{proof}

Now, we assume that $\deg f> \deg g$. Then $d_t$ and $V^t_{\infty}$ do not depend on $t$, set $d:=d_t$. The following is deduced from Corollary 4.4 in \cite{D}.

\begin{proposition} \label{prop4.2.14} For all $t\in \Bbb{C}$, we have
$$\chi(\overline{V_t})= \chi(V) + \Sigma_{p \in\, \textrm{Sing} (V_t)} \mu_p(G(x, y, z, t)),$$
where $V$ is a smooth projective curve of degree $d$ and $\overline{V_t}$ is the projective closure of the curve $V_t\subset \mathbb{C}^2$.
\end{proposition}

\begin{lemma}\label{lm4.2.13} Let $F=f/g: \Bbb{C}^2\setminus\{g= 0\}\to \Bbb{C}$ be a rational function and $t_0$ be a regular value of $F$. Assume that $\deg f> \deg g$ and
$\chi(F^{-1}(t)) = \chi(F^{-1}(t_0))$ for all $t$ near $t_0$ enough. Then $t_0\notin K_1(F)$ and there exists a neighborhood $D$ of $t_0$ such that
$$\mu_p(G(x, y, z, t))= \mu_p(G(x, y, z, t_0)), p\in V_{\infty}, t\in D.$$ 
\end{lemma}
\begin{proof} Let $D$ be a neighborhood of $t_0$ such that  $\chi(V_t) = \chi(V_{t_0})$ for all $t\in D.$ Since $t_0\notin K_0(F)$, we can choose $D$ small enough such that $F^{-1}(t)$ is smooth for all $t\in D$.

By using the Mayer-Vietoris exact sequence, we obtain
$$\chi(F^{-1}(t))= \chi (\bar{V_t})- \# V_{\infty} - \#A(F).$$
  Therefore, according to Proposition \ref{prop4.2.14}, we have
\begin{align*}
&\chi (F^{-1}(t)) - \chi (F^{-1}(t_0))= \Sigma \mu _p( G(x, y, z, t)) -  \Sigma \mu _q (G(x, y, z, t_0))\\
& = \Sigma_{p \in A(F)} (\mu_p (f-tg) - \Sigma \mu_p (f-t_0g) ) + \Sigma_{p \in V_{\infty}} (\mu_p (G(x, y, z, t)) - \\
&-\mu_p (G(x, y, z, t_0))).
\end{align*}
Since the Milnor number is a semi-continuous function in $t$, then $\chi (F^{-1}(t)) - \chi (F^{-1}(t_0))\leq 0$, the equality occurs if and only if 
 $$\mu_p (G(x, y, z, t)) = \mu_p (G(x, y, z, t_0)),$$
 for all $p\in V_{\infty}, p \in A(F).$
\end{proof} 

\begin{proof}[Proof of Theorem \ref{thm12}] According to Lemma \ref{thrmbifur-huuty1}, it is enough to prove that 
$$K_0(F)\cup B_{\infty}(F)\cup K_1(F)\subset B(F).$$
Let $t_0\notin B(F)$ arbitrary. Then $F$ defines a locally $C^{\infty}-$trivial fibration at $t_0$. Let $D$ be the neighborhood of $t_0$ such that the restriction
$$F_{|F^{-1}(D)} \colon F^{-1}(D)\to D$$
 is a $C^{\infty}-$trivial fibration. That implies $t_0\notin B_{\infty}(F).$

According to the Sard's Theorem, we can take a regular value $t_1$ of $F_{|F^{-1}(D)}$. Therefore the fiber $F^{-1}(t_1)$ is smooth. Since $F_{|F^{-1}(D)}$ is trivial,  it is also smooth. Thus $t_0\notin K_0(F).$

On the other hand, for all $t\in D$ the fiber $F^{-1}(t_0)$ is homeomorphic to $F^{-1}(t)$. Therefore their Euler characteristic are equal, by Lemma   \ref{lm4.2.13}, we get $t_0\notin K_1(F)$. The proof is complete.
\end{proof}

\section{Critical values at infinity}
Let $F=\frac{f}{g}: \Bbb{C}^2\setminus\{g= 0\}\to \Bbb{C}$ be a rational function, where $f, g\in \mathbb{C}[x, y]$ have no common factor. This section is to characterize the critical values at infinity of $F$.

Let $t_0\in \mathbb{C}\setminus (K_0(F)\cup K_1(F))$ such that
\begin{align}
d:= \deg (f-t_0g)= \max \{ \deg f, \deg g\}.\label{dkbac}
\end{align}
Without loss of generality, we may assume that
$$d=\textrm{deg}_x (f-t_0g).$$ 

\begin{remark}{\rm 
The assumption \ref{dkbac} holds in the following situations:

1) $\deg f> \deg g;$

2) $\deg g> \deg f$ and $t_0\neq 0;$

3) $\deg f= \deg g= d$ and $t_0\neq \frac{f_d}{g_d}$, where $f_d, g_d$ are respectively the highest-degree homogeneous components of $f, g$.
}
\end{remark}
\subsection{Geometrical and topological characterizations}

We denote by $L$ the following linear function
$$\Bbb{C}^2\to \Bbb{C}, (x, y) \mapsto y.$$
For each $t\in \Bbb{C}$ let
$$L_t:=L_{|V_t}: V_t\to \Bbb{C}$$
and
$$l_t:=L_{|F^{-1}(t)}: F^{-1}(t)\to \Bbb{C},$$
where $V_t= \{(x, y)\in \mathbb{C}^2 : f(x, y)- tg(x, y)=0\}$.
It is easy to prove that

\begin{lemma}\label{lm3.1}\label{lm4.2.2} For all $\delta > 0$ small enough and $t \in D_{\delta}(t_0)$, the map
$$L_t: V_t\to \Bbb{C}$$
is proper and $ \# L_t^{-1}(c)=d$, where $c$ is a generic constant and $D_{\delta}(t_0)=\{t\in \Bbb{C}: |t-t_0|< \delta\}.$
\end{lemma}

The following follows from Lemma \ref{lm3.1} and the argument in the proof of Lemma 3.2 in \cite{HT}.

\begin{lemma}\label{lm3.2} Under the hypothesis in Lemma \ref{lm4.2.2}, for all $\delta>0$ small enough, the restriction
$$L_{\delta}:=L_{| \cup_{t\in \overline{D_{\delta}(t_0)}} V_t}: \cup_{t\in \overline{D_{\delta}(t_0)}} V_t\to \Bbb{C}$$
is proper.
\end{lemma}

\begin{remark}\label{rm4.1}{\rm
1) The critical points of $l_t:F^{-1}(t)\to \Bbb{C}$ are exactly the critical points of $L_t: V_t\to \Bbb{C}$ not belonging to the set $V_t$.

2) The critical points of $L_t: V_t\to \Bbb{C}$ are algebraic functions in $t$, we can divide them into two types:

\hskip 1cm (i) The points which tend to critical points of $L_{t_0}$ as $t\to t_0$. The number of points in this type, counting with multiplicity, is equal to the number of critical points, counting with multiplicity, of $L_{t_0}$.

\hskip 1cm (ii) The points that tend to infinity as $t\to t_0$ (the points in this type are also critical points of $l_t
$).

}\end{remark}

\begin{lemma}\label{lm4.2.6}\label{lm3.4}  For each $a>0$ and $\delta>0$ let 
$$U(a, \delta) :=\{|L|\leqslant a\}\cap \overline{F^{-1}(D_{\delta}(t_0))}.$$
For $a$ is large enough and $\delta$ is small enough we have
\begin{align*}
\chi(V_t)- \chi(V_{t_0})&= \chi(F^{-1}(t_0))-\chi(F^{-1}(t))\\
&= \chi(F^{-1}(t_0)\setminus U(a, \delta))- \chi(F^{-1}(t)\setminus U(a, \delta)),
\end{align*}
where $V_t= \{(x, y)\in \mathbb{C}^2 : f(x, y) - tg(x, y)=0\}.$
\end{lemma}

\begin{proof}
According to Remark \ref{rm4.1}, for sufficiently small enough $\delta$ and sufficiently large enough $a$, all critical points $L_t$ in $U(a, \delta)$ are in the first type. Let $Q_i, i=1,\dots, s$ be the critical points of $L_{t_0}$ in $U(a, \delta)$. Let $D_{\beta_i}$ as the disc centered at $L(Q_i)$, with the radius sufficiently small enough $\beta_i$, then for sufficiently small enough $\delta $ there is no critical point of $L_t$ in $U(a, \delta)\setminus \cup_{i=1}^s L^{-1}(D_{\beta_i})$. 

For each  $i=1,\dots, s$ and $t\in \mathbb{C}$ we denote
$$M_t^i:= (V_t\cap U(a, \delta))\setminus L_{\delta}^{-1}(D_{\beta_i}),\, N_t^i:= (V_t\cap U(a, \delta))\cap L_{\delta}^{-1}(D_{\beta_i})$$
and
$$C^i:= \{z\in \Bbb{C}: |z|\leqslant a\}\setminus D_{\beta_i}.$$

According to Lemma \ref{lm3.1}, for all $t\in D_{\delta}(t_0)$ the restriction $L_t= L_{|V_t}$ is proper, then $L(V_t)$ is close and constructible. Hence $L(V_t)= \mathbb{C}$ and the restriction map $L_{|M_t^i}\colon M_t^i\to C^i$ is surjective. Moreover, it is easy to check that $L_{|M_t^i}$ does not have critical point. Thus, the map $L_{|M_t^i}$ is an $d$-sheeted unbranched covering. Then
\begin{eqnarray}
\chi(M_t^i)= d\cdot \chi(C^i), \forall t\in D_{\delta}(t_0).\label{dangthucEuler1}
\end{eqnarray}
On the other hand, the restriction map $L_{|N_t^i}: N_t^i\to D_{\beta_i}$ 
is a $d$-sheeted covering branching over the critical points. By the same argument as in the proof of Theorem 3.1 in \cite{HT}, we have
\begin{eqnarray}
\chi(N_t^i)=d-\rho_i(t),  \forall t\in D_{\delta}(t_0),\label{dangthucEuler2}
\end{eqnarray}
where $\rho_i(t)$ is the number of critical points in $D_{\beta_i}$, counting with multiplicity, of  $L_t.$
By using the Mayer-Vietoris exact sequence, we get
$$\chi(V_t)-\chi(V_{t_0})= \left(\chi(V_t\setminus U(a, \delta))- \chi(V_{t_0}\setminus U(a, \delta))\right) +  \sum_{i=1}^s (\rho_i(t_0)-\rho_i(t)).$$
Moreover, by Remark \ref{rm4.1}, the second term is equal to $0$ and $V_t\setminus U(a, \delta)= F^{-1}(t)\setminus U(a, \delta)$. Then
$$\chi(V_t)-\chi(V_{t_0})= \chi(F^{-1}(t)\setminus U(a, \delta))- \chi(F^{-1}(t_0)\setminus U(a, \delta)).$$

Similarly, since $t_0\notin K_1(F)$, by using the Mayer-Vietoris exact sequence again, we can prove that
$$\chi(V_t)-\chi(V_{t_0})=\chi(F^{-1}(t)) - \chi(F^{-1}(t_0)).$$
From the last two equalities, we get the conclusion of the lemma.
\end{proof}

\begin{theorem}\label{thrm4.1}\label{thrm4.2.7}
Let $F= f/g: \Bbb{C}^2\setminus\{g= 0\}\to \Bbb{C}$ be a rational function, where $f, g\in \mathbb{C}[x, y]$ have no common factor, and let $t_0\notin K_0(F)\cup K_1(F)$ such that
$$\deg(f-t_0g)= \deg_x(f-t_0g)= \max\{\deg f, \deg g\}.$$
Then the followings are equivalent:

(i) $t_0\notin B_{\infty}(F);$

(ii) There is no critical point of $l_t=L_{|F^{-1}(t)}$ which tends to infinity as $t\to t_0$.
\end{theorem}

\begin{proof}ii) $\Longrightarrow $ i): Assume that there is no critical point of $l_t$ going to infinity as $t\to t_0$. It follows that if a number $a$ is large enough, then the set
$$U(a)=\{|L|\leqslant a\}\cap (\overline{F^{-1}(D_{\delta}(t_0))})$$
contains all the critical points of the maps $l_t, t\in D_{\delta}(t_0)$. It follows from Lemma \ref{lm3.2} that $U(a)$ is bounded, hence, is a compact set.  By the same argument as in proof of Theorem 3.1 in \cite{HT}, the restriction
$$F_{|F^{-1}(D_\delta(t_0))\setminus U(a)}: F^{-1}(D_\delta(t_0))\setminus U(a)\to D_\delta(t_0)$$
is a trivial fibration. Hence $t_0\notin B_{\infty}(F).$

i) $\Longrightarrow $ ii): By contradiction, assume that there exist critical points of $l_t$ going to infinity as $t\to t_0$. 

Let 
$$K :=U(a)=\{|L|\leqslant a\}\cap \overline{F^{-1}(D_\delta(t_0))},$$
 where $|a| \gg 1$ such that all critical point of $l_{t_0}$, all critical points in the first type of $L_t, t\in D_{\delta}(t_0)$  and the points of the set $A(F)$ are contained in $K$. It follows from the assumption that for arbitrarily small $\delta$, there exists $t\in D_\delta(t_0)$ such that $l_t$ has critical points $P_1(t), \dots, P_m(t)$ that do not belong to $K$.

Let $D_{\epsilon_i}, i=1, \dots, m,$ be the disc centered at $\alpha_i := L(P_i(t))$ with radius $\epsilon_i$ small enough. We consider the following restrictions
$$L: (F^{-1}(t_0)\setminus K)\setminus \cup_{i=1}^{m}l_{\delta}^{-1}(D_{\epsilon_i})\to (\Bbb{C}\setminus l_{\delta}(K))\setminus \cup_{i=1}^{m}D_{\epsilon_i}$$
and
$$L: (F^{-1}(t)\setminus K)\setminus \cup_{i=1}^{m}l_{\delta}^{-1}(D_{\epsilon_i})\to (\Bbb{C}\setminus l_{\delta}(K))\setminus \cup_{i=1}^{m}D_{\epsilon_i}.$$
These maps are well-defined. By the similar arguments as in Proof of Lemma \ref{lm3.4}, we can prove that these maps are $d-$sheeted unbranched coverings. As a consequence, we have the following
\begin{align*}
\chi((F^{-1}(t_0)\setminus K)\setminus \cup_{i=1}^{m}l_{\delta}^{-1}(D_{\epsilon_i}))&=\chi((F^{-1}(t)\setminus K)\setminus \cup_{i=1}^{m}l_{\delta}^{-1}(D_{\epsilon_i}))\\
&= d\chi((\Bbb{C}\setminus l_{\delta}(K))\setminus \cup_{i=1}^{m}D_{\epsilon_i}).
\end{align*}
Now, for each $i=1, \ldots, m$ and $t\in D_{\delta}(t_0)$ let us consider the restricted map
$$L^i_t:= L_{|F^{-1}(t)\cap L^{-1}(D_{\epsilon_i})}: F^{-1}(t)\cap L^{-1}(D_{\epsilon_i})\to D_{\epsilon_i}.$$
Since $V_t =F^{-1}(t)\cup A(F),$ we have $\Bbb{C}\setminus L(K)\subset L(F^{-1}(t))$. Hence, we can choose $\epsilon_i$ small enough such that $L^i_t$ is surjective. Moreover, for all $t$ the map $L^i_t$ is proper, for $t \neq t_0$ the map $L^i_t$ has critical points $P_i(t)$, and $L^i_{t_0}$ has no critical points, then by using the same argument as in the proof of Theorem 3.1 in \cite{HT}, we have
$$\chi(F^{-1}(t_0)\cap L^{-1}(D_{\epsilon_i}))=d,$$
and
$$\chi(F^{-1}(t)\cap L^{-1}(D_{\epsilon_i}))=d-r_i,$$
where $r_i$ is the multiplicity of the critical point $P_i(t)$ of the map $l_t$.

It follows from the Mayer-Vietoris sequence that
$$\chi(F^{-1}(t_0)\setminus K)- \chi(F^{-1}(t)\setminus K)=\sum_{i=1}^m r_i\neq 0,$$
since there are critical points of $l_t$ tending to infinity when $t$ tends to $t_0$. By applying Lemma \ref{lm3.4}, we get $\chi(F^{-1}(t))\neq \chi(F^{-1}(t_0))$ for all $t$ near $t_0$. Thus $t_0\in B_{\infty}(F).$
\end{proof}

Let $\delta (y, t) = \textrm{disc}_x(f-tg)$ be the discriminant of $f-tg$ with respect to $x$. Then the critical points of $l_t$ are $(x(t), y(t))$ such that $y(t)$ is a root of $\delta(y, t)=0$. These  points go to infinity as $t\to t_0$ if and only if $y(t)\to \infty$ when $t\to t_0$. We can write
$$\delta(y, t) = q_k(t)y^k + q_{k-1}y^{k-1} + \cdots.$$
Then $\delta (y, t)$ has a root tending to infinity when $t\to t_0$ if and only if $q_k(t_0) = 0.$ The following is an immediate corollary of Theorem \ref{thrm4.1}.

 \begin{corollary}
Let $F= f/g: \Bbb{C}^2\setminus\{g= 0\}\to \Bbb{C}$ be a rational function, where $f, g\in \mathbb{C}[x, y]$ have no common factor and $t_0\notin K_0(F)\cup K_1(F)$ such that
$$\deg(f-t_0g)= \deg_x(f-t_0g)= \max\{\deg f, \deg g\}.$$
Then, $t_0\in B_{\infty}(F)$ if and only if $q_k(t_0) = 0.$
\end{corollary}

\begin{theorem}\label{thrm2.1}\label{thrm4.2.9}
Let $F= f/g: \Bbb{C}^2\setminus\{g= 0\}\to \Bbb{C}$ be a rational function, where $f, g\in \mathbb{C}[x, y]$ have no common factor, and let $t_0\notin K_0(F)\cup K_1(F)$ such that
$$\deg(f-t_0g)= \deg_x(f-t_0g)= \max\{\deg f, \deg g\}.$$
Then the followings are equivalent:

(i) $t_0\notin B_{\infty}(F);$

(ii) There exists a compact subset $K$ of $\Bbb{C}^2$ such that
$$\chi(F^{-1}(t_0)\setminus K)> \chi(F^{-1}(t)\setminus K),$$ 
for all $t$ generic.
\end{theorem}
\begin{proof}
(i) $\Longrightarrow $ (ii): Let $K:=U(a, \delta)=\{|L|\leqslant a\}\cap \overline{F^{-1}(D_{\delta}(t_0))}.$ According to the proof of Theorem \ref{thrm4.2.7}, if $a$ is large enough and $\delta$ is small enough then
$$\chi(F^{-1}(t_0)\setminus K)- \chi(F^{-1}(t)\setminus K)= \rho, \forall t\in D_{\delta}(t_0),$$
where $\rho$ is the number of critical points  $P(t)$, counting with multiplicity, of the map $l_t$ such that $\|P(t)\|\to \infty$ as $t\to t_0.$ 

Since $t_0\in B_{\infty}(F)$ then according to Theorem \ref{thrm4.2.7}, we have $\rho \neq 0.$ Thus
$$\chi(F^{-1}(t_0)\setminus K)> \chi(F^{-1}(t)\setminus K),$$
for all $t$ different and near $t_0$.

(ii) $\Longrightarrow $ (i): Let $K$ be the compact set such that
$$\chi(F^{-1}(t_0)\setminus K)> \chi(F^{-1}(t)\setminus K).$$
By contradiction, assume that $t_0\notin B_{\infty}(F)$. Since $t_0\notin K_0(F)\cup K_1(F)$ then $F$ defines a locally $C^{\infty}-$trivial fibration at $t_0$. Let $D$ be the neighbourhood of $t_0$ and $\Phi: F^{-1}(D)\to F^{-1}(t_0)\times D$ be the diffeomorphism trivializing $F_{|F^{-1}(D)}$.

Hence, the restriction $\Phi_{|F^{-1}(D)\setminus K}$ induces a diffeomorphism trivializing the map $$F_{|F^{-1}(D)\setminus K}\colon F^{-1}(D)\setminus K\to D.$$ This implies that
$$\chi(F^{-1}(t)\setminus K)= \chi(F^{-1}(t_0)\setminus K)$$
which contradicts the assumption. Thus $t_0\in B_{\infty}(F)$.
\end{proof}

\begin{theorem}\label{thrm2.2}\label{thrm4.2.10}
Let $F= f/g: \Bbb{C}^2\setminus\{g= 0\}\to \Bbb{C}$ be a rational function, where $f, g\in \mathbb{C}[x, y]$ have no common factor, and let $t_0\notin K_0(F)\cup K_1(F)$ such that
$$\deg(f-t_0g)= \deg_x(f-t_0g)= \max\{\deg f, \deg g\}.$$
Then the followings are equivalent:

(i) $t_0\in B_{\infty}(F);$

(ii) $\chi(\{f-t_0g= 0\})> \chi(\{f-tg= 0\})$, for all $t$ generic;

(iii) $\chi(F^{-1}(t_0))> \chi(F^{-1}(t))$, for all $t$ generic.
\end{theorem}

\begin{proof}
The proof is straightforward from Lemma \ref{lm4.2.6} and the proof of Theorem \ref{thrm4.2.7}.
\end{proof}

\subsection{Analytic characterization}

In this section, we will determine the critical values at infinity of rational functions in two variables in terms of Malgrange condition and M-tameness.
Assume that $d:= \deg{f}> \deg{g}.$

\begin{definition}
{\rm (\cite{LS}, \cite{P})  Let $H(t, x): \Bbb{C}^{n+1}\to \Bbb{C}$ be a analytic function such that for every $t$ the point $0\in \Bbb{C}^n$ is an isolated singular point of $H(t, x)$. Then the following set
$$\Gamma_{H} = \overline{\{(t, x)\in  \Bbb{C}^{n+1}: \partial H/\partial t \neq 0, \partial H/\partial x_1 = \ldots= \partial G/\partial x_n = 0\}}$$
 is called the {\it relative polar curve} of the family of hypersurface $\{x\in \mathbb{C}^n : H(t, x) = 0\}$.
}
\end{definition}

\begin{theorem}{\rm (\cite{LS}, \cite{P})}\label{thrmLS} Let $H(t, x): \Bbb{C}^{n+1}\to \Bbb{C}$ be an analytic function such that for every $t$ near $t_0$ enough the origin is an isolated singular point of $H_t(x):= H(t, x)$. Then, the followings are equivalent

(i) $|\partial H/\partial t(t, x)|\ll \|(\partial H/\partial x_1, \ldots, \partial G/\partial x_n)(t, x)\|$
for all $(t, x)$ near $(t_0, 0)$ enough; 

(ii)  $\mu_0(H(t, x)) = \mu_0(H(t_0, x))$ for all $t$ near $t_0$ enough;

(iii) There exist a neighbourhood $B$ of $(t_0, 0)$ such that $\Gamma_H\cap B= \emptyset$.
\end{theorem}

Since $ \deg{f}> \deg{g}$ then $\deg (f-tg)= \deg f$ for all $t$ and the set $V^t_{\infty}$ of points at infinity of $V_t$ does not depend on $t$. Denote  $V_{\infty}:= V^t_{\infty}$.
 
Considering a point $p_0\in V_{\infty}$. Without loss of generality, we may assume that $p_0= [1 : 0 : 0] \in \mathbb{C}P^2$. Then $(y, z)$ forms a local system of coordinates near $p_0$. Let $G(y,z, t) := G(1, y, z, t)$. Then either $(0, 0)$ is a regular point or an isolated singular point of $G(y, z, t)$.

The following is a version of Lemma 3.1 in \cite{P} for rational functions.

\begin{lemma}\label{lm4.2.17}\label{lm4.2} Let $F=f/g\colon \Bbb{C}^2\setminus\{g= 0\}\to \Bbb{C}$ be a rational function, where $f, g\in \Bbb{C}[x, y]$ and $\deg f> \deg g.$ Let $t_0\in \Bbb{C}\setminus (K_0(F)\cup K_1(F))$ and  $p_0 = [1, 0, 0]\in V_{\infty}$. Assume that, either $p_0$  is a regular point of $\overline{V_{t_0}}$ or  $p_0$  is a singular point of  $\overline{V_{t_0}}$ and $\mu_{(0, 0)}(G(y, z, t)) = \mu_{(0, 0)}(G(y, z, t_0))$ for all $t$ near $t_0$ enough.

Then, for all positive integer $N$, the following holds
\begin{align}
|\partial G/\partial t|\ll \|(\partial G/\partial y, z^{(N-1)/N}\partial G/\partial z)(y, z, t)\|\label{bdta_G}
\end{align}
as $(y, z, t)\to (0, t_0).$
\end{lemma}

\begin{proof}The case that $p_0$ is nonsingular is easy. Now, we assume that for each $t$, $p_0$ is a singular point of $\bar{V_t}$  and $\mu_{(0, 0)}(G(y, z, t)) = \mu_{(0, 0)}(G(y, z, t_0))$ for all $t$ near $t_0$ enough. For each $N>1$, we consider the function
$$G_N(y, z, t) = G(y, z^N, t).$$
Then $\frac{\partial G_N}{\partial y}= \frac{\partial G}{\partial y}, \frac{\partial G_N}{\partial z} = Nz^{N-1}\frac{\partial G}{\partial z}$ and $\frac{\partial G_N}{\partial t}= \frac{\partial G}{\partial t}.$

Since $(0, 0)$ is a singular point of $\{G(y, z, t_0)= 0\}$, it is easy to prove that $(0, 0)$ is also an isolated singular point of $\{G_N(y, z, t) =0\}$ for all $t$.

According to Theorem \ref{thrmLS}, it suffices to show that the relative polar curve:
$$\Gamma_{G_N} = \overline{\{(y, z, t) | \partial G_N/\partial t \neq 0, \partial G_N/\partial y = 0, \partial G_N/\partial z = 0\}}$$
  of the family $\{G_N(y, z, t)= 0\}$ is empty in some small neighborhood of $(0, 0, t_0)$. 

By contradiction, we assume that there exists a sequence $(y_k, z_k, t_k)\to (0, 0, t_0)$ such that
\begin{align} 
\frac{\partial G_N}{\partial t}(y_k, z_k, t_k)\neq 0, \frac{\partial G_N}{\partial y}(y_k, z_k, t_k) = \frac{\partial G_N}{\partial z}(y_k, z_k, t_k)= 0, \forall k. \label{bdta_GN}
\end{align}

We have 
$$\frac{\partial G_N}{\partial t}(y_k, z_k, t_k)= \frac{\partial G}{\partial t}(y_k, z_k, t_k)= -z_k^dg(1/z_k, y_k/z_k)\neq 0.$$
 Since $d> \deg g$ then $z_k\neq 0$. Hence, it follows from (\ref{bdta_GN}) that $\frac{\partial G}{\partial z}(y_k, z_k, t_k)= 0.$ 

That means the relative polar curve  $\Gamma_G$ of the family $\{G(y, z, t)= 0\}$ is not empty in some neighborhood of $(0, 0, t_0).$ This contradicts to the assumption. The proof is complete.
\end{proof}

By the Curve Selection Lemma and by using the inequality in Lemma \ref{lm4.2} for all $N$, we obtain
\begin{align}
|\partial G/ \partial t| \leqslant C \|(\partial G/ \partial y, z \partial G/ \partial z)\|.\label{bdt_aGN2}
\end{align}
By applying the same argument as in the proof of Lemma 3.2 in \cite{P}, we receive the following.

\begin{lemma}\label{lm4.3}
Under the hypothesis in Lemma \ref{lm4.2}, for all $(y, z, t) \in \{G(y, z, t) = 0\}$ and $(y, z, t)\to (0, t_0)$, we have
\begin{align}
|z \partial G/ \partial z| \ll  |\partial G/ \partial y|.\label{bdt_aGN3}
\end{align}
\end{lemma}

\begin{theorem}\label{thm4.1} Let $F=f/g\colon \Bbb{C}^2\setminus\{g= 0\}\to \Bbb{C}$ be a rational function, where $f, g\in \Bbb{C}[x, y]$ and $\deg f> \deg g.$ Let $t_0\in \Bbb{C}\setminus (K_0(F)\cup K_1(F)\cup B_{\infty}(F))$. Then $F$ satisfies the Malgrange's condition at $t_0$.
\end{theorem}

\begin{proof}
Let $D$ be the neighbourhood of $t_0$ such that
$$\chi (F^{-1}(t_0))= \chi (F^{-1}(t)), \forall t\in D.$$
 According to Lemma \ref{lm4.2.13}, for all $p\in V_{\infty}$ and $t\in D$, either $p$ is a regular point of $\bar{V_{t_0}}$ or is a singular point and $\mu_p(G(x, y, z, t))= \mu_p(G(x, y, z, t_0))$. 

Let $p\in V_{\infty}$  arbitrary. Without loss of generality, we may assume that $p= [1, 0, 0]$. It follows from inequalities (\ref{bdt_aGN2}) and (\ref{bdt_aGN3}) that
$$|\partial G/ \partial t| \leqslant C |\partial G/ \partial y(y, z, t)|$$
for all $(y, z, t) \in \{G(y, z, t) = 0\}$ near $(0, 0, t_0)$ enough, where $G(y, z, t) = z^d (f(1/z, y/z) - tg(1/z, y/z)).$ We have
$$|z\cdot g(1/z, y/z)|\leqslant |\frac{\partial }{\partial Y}(f-tg)(1/z, y/z)|$$
for all $(y, z, t)$ near $(0, 0, t_0)$ enough such that $(f-tg)(1/z, y/z)=0$.

Now, set $z=\frac{1}{X}$ and $y= \frac{Y}{X}$, we obtain
$$0< 1/C \leqslant \|(X, Y)\|\cdot \|\grad F(X, Y)\|$$
for all $(X, Y)\to \infty$ and $F(X, Y)\to t_0$. Thus $F$ satisfies the Malgrange's condition at $t_0$.
\end{proof}
Now we consider the M-tameness of $F$. Firstly, we prove the following.

\begin{theorem}\label{thrm4.2}
If $F$ satisfies the Malgrange's condition at a value $t_0$, then $F$ is M-tame at $t_0$.
\end{theorem}

\begin{proof}
Assume that $F$ is not M-tame at $t_0$, i.e. there are sequences $\{p_k\}_k$ and $\{\lambda_k\}_k$ such that
$$p_k \rightarrow \infty, F(p_k) \rightarrow t_0 \, \,  \textrm{and} \, \,  \grad F(p_k)=\lambda _k p_k.$$
We will show that $F$ does not satisfy the Malgrange's condition at $t_0$.

Indeed, By the Curve Selection Lemma, there exist some real analytic curves $(x(\tau), y(\tau))\to \infty$ and $\lambda(\tau), \tau \in (0, \epsilon)$ such that $ \mathrm{grad}F(x(\tau), y(\tau))= \lambda(\tau)(x(\tau), y(\tau))$ and $F(x(\tau), y(\tau))\to t_0$ when $\tau \to 0$. 

For each analytic curve $\phi(\tau) = c \tau^m +$ {\it higher powers} ($c\neq 0$),  we denote $\deg(\phi(\tau))= m.$  If $\phi(\tau)$ and $\rho (\tau)$ are two analytic curves then we define  $\deg(\frac{\phi(\tau)}{\rho (\tau)}) = \deg(\phi(\tau)) - \deg(\rho (\tau))$. 

Since $F(x(\tau), y(\tau))\to t_0$, then $\deg(f(x(\tau), y(\tau))) = \deg(g(x(\tau), y(\tau)))$. Hence $\deg F^{'}(x(\tau), y(\tau)) > -1$. Thus
$$\deg (<(x(\tau), y(\tau)), \mathrm{grad} F >) > 0 \ \ \textrm{ and }\ \ \ \|(x(\tau), y(\tau)) \|\cdot \|\mathrm{grad} F\| \to 0.$$
Therefore $F$ does not satisfy the Malgrange's condition at $t_0$.

\end{proof}

The main theorem in this section is the following.

\begin{theorem}\label{main2}
Let $F=f/g\colon \Bbb{C}^2\setminus\{g= 0\}\to \Bbb{C}$ be a rational function, where $f, g\in \Bbb{C}[x, y]$ and $\deg f> \deg g.$ Let $t_0\in \Bbb{C}\setminus (K_0(F)\cup K_1(F))$. Then, the followings are equivalent

(i) $t_0 \in B_{\infty}(F);$

(ii)  $t_0 \in K_\infty(F);$

(iii) $t_0 \in M_{\infty}(F).$

\end{theorem}

\begin{proof}
(i) $\Longrightarrow $ (iii): Assume that $F$ is M-tame at $t_0$. Then for $\delta >0$ small enough and $R>0$ large enough, we can construct in $F^{-1}(D_{\delta}(t_0))\setminus B_R$ a smooth vector field $v(x)$ such that
$$
\begin{aligned}
<& v, x> = 0;\\
<& v, \mathrm{grad} F>= 1.
\end{aligned}
$$
Where $B_R$ is the close ball in $\mathbb{C}^2$ with radius $R$ centered at the origin. 

Now, by integrating the vector field  we get a diffeomorphism trivializing the map
$$F: F^{-1}(D_\delta(t_0))\setminus B_R \to D_\delta(t_0).$$

(ii) $\Longrightarrow $ (i): By Theorem \ref{thm4.1}. 

(iii) $\Longrightarrow $ (ii): By Theorem \ref{thrm4.2}.
\end{proof}

\begin{remark}
{\rm Theorem \ref{main2} remains valid if $\deg f< \deg g$ and $t_0\neq 0$.
}
\end{remark}
\subsection{Examples}
To conclude we give some examples showing that the Fedoryuk condition is not necessary for a value to be regular at infinity. 

\begin{example}{\rm
Let $F(x, y)=\frac{xy+1}{x^2+1}$ and $L:\Bbb{C}^2\to \Bbb{C}, (x, y)\mapsto y$. 

The critical points of $l_t$ are $(x, y)$, where $x=\sqrt{(1-1/t)}, y=2x/(1-x^2)$. It is easy to check that these points do not go to infinity when $t\to i$. According to Theorem \ref{thrm4.1} we have $i \notin B_{\infty}(F).$

Now let $(x_k, y_k)=(k, ik)$. We see that $\|(x_k, y_k)\|\to \infty, F((x_k, y_k))\to i$ and $\|\textrm{grad}F(x_k, y_k)\|\to~ 0$ as $k \to \infty.$
That means $i\in \widetilde{K}_\infty(F).$ Thus $\widetilde{K}_\infty(F)\not\subset B_{\infty}(F).$}
\end{example}

\begin{example}{\rm
Let $F(x, y)=\frac{x^3+1}{xy+1}.$
It is easy to check that $B_{\infty}(F)= \{0\}.$ Let $p_k := (k, \frac{1}{c}k^2) \in {\Bbb C}^2, c \ne 0, k \ge 1,$ we see that $F(p_k) \to c$ as $k\to \infty,$ and $\textrm{grad}F(p_k)\to 0.$ Therefore $\widetilde{K}_\infty(F) = \mathbb{C}.$ In particular, $B_{\infty}(F) \neq \widetilde{K}_\infty(F)$.
}
\end{example}

\section*{Acknowledgments} The author would like to thank Professor Ha Huy Vui for useful discussions. The author would like to thank the referee(s) for carefully examining our paper and providing many valuable comments.

\end{document}